\documentclass{article}
\usepackage{amssymb,amsmath,amsthm, hyperref}
\usepackage[latin1]{inputenc}

\usepackage{esint}
\usepackage{ulem}
\usepackage[mathscr]{eucal}
\usepackage{stackrel}
\usepackage[nottoc,notlot,notlof]{tocbibind}
\usepackage[usenames,dvipsnames,svgnames,table]{xcolor}
\usepackage{hyperref}

\usepackage[usenames,dvipsnames]{pstricks}
 \usepackage{epsfig}
 \usepackage{pst-grad} 
 \usepackage{pst-plot} 

\newtheorem{theorem}{Theorem}
\newtheorem{definition}{Definition}
\newtheorem{lemma}{Lemma}
\newtheorem{proposition}{Proposition}
\newtheorem{corollary}{Corollary}
\newtheorem{remark}{Remark}
\newtheorem{example}{Example}
\date{}
\numberwithin{equation}{section} \numberwithin{theorem}{section}
\numberwithin{lemma}{section} \numberwithin{corollary}{section}
\numberwithin{remark}{section} \numberwithin{proposition}{section}
\numberwithin{definition}{section}

\newcommand{\loc}{\operatorname{loc}}

\def \N {\mathbb{N}}
\def \R {\mathbb{R}}
\def \supp {\mathrm{supp } }
\def \diam {\mathrm{diam}}

\def \dist {\mathrm{dist}}

\def \loc {\mathrm{loc}}

\def \suchthat {\ \big | \ }

\def \tr {\mathrm{Tr}}

\def \e {\varepsilon}

\newcommand{\pe}{E_{\varepsilon}}
\newcommand{\defeq}{\mathrel{\mathop:}=}

\newcommand{\intav}[1]{\mathchoice {\mathop{\vrule width 6pt height 3 pt depth  -2.5pt
\kern -8pt \intop}\nolimits_{\kern -6pt#1}} {\mathop{\vrule width
5pt height 3  pt depth -2.6pt \kern -6pt \intop}\nolimits_{#1}}
{\mathop{\vrule width 5pt height 3 pt depth -2.6pt \kern -6pt
\intop}\nolimits_{#1}} {\mathop{\vrule width 5pt height 3 pt depth
-2.6pt \kern -6pt \intop}\nolimits_{#1}}}

\begin{document}

\title{{\bf Cavity type problems ruled by infinity Laplacian operator}}

\author{\it \smallskip \\
G. C. Ricarte\footnote{\noindent \textsc{Gleydson Chaves Ricarte} --
Universidade Federal do Cear\'{a}, Av. Humberto Monte s/n, Campus do Pici - Bloco 914, 60455-760 Fortaleza, CE, Brazil.
\texttt{E-mail: ricarte@mat.ufc.br}
},\,\,\,\,J. V. Silva\footnote{\noindent \textsc{Jo\~{a}o V\'{i}tor da Silva} -- Universidad de Buenos Aires, Departmento de Matem\'{a}tica, Ciudad Universitaria-Pabell\'{o}n I-(C1428EGA) - Buenos Aires, Argentina. \noindent \texttt{E-mail: jdasilva@dm.uba.ar}
}\,\,\hbox{ and }
  R. Teymurazyan\footnote{\noindent \textsc{Rafayel Teymurazyan} -- Universidade de Coimbra, Departmento de Matem\'{a}tica, 3001-501 Coimbra, Portugal. \texttt{E-mail: rafayel@utexas.edu}
}}

\maketitle

\begin{abstract}
We study a singularly perturbed problem related to infinity Laplacian operator with prescribed boundary values in a region. We prove that solutions are locally (uniformly) Lipschitz continuous, they grow as a linear function, are strongly non-degenerate and have porous level surfaces. Moreover, for some restricted cases we show the finiteness of the $(n-1)$-dimensional Hausdorff measure of level sets. The analysis of the asymptotic limits is carried out as well.



\end{abstract}

\section{Introduction}

\hspace{0.3cm} In this paper we study inhomogeneous singularly perturbed problems ruled by the \textit{Infinity Laplacian}, which is defined as follows:
$$
\displaystyle \Delta_{\infty} u \defeq (Du)^TD^2uDu = \sum_{i,j=1}^n u_{x_i} u_{x_j} u_{x_i x_j}.
$$
More precisely, we study {\it weak} solutions to
\begin{equation}\label{Equation Pe} \tag{$\pe$}
\left\{
\begin{array}{rcccl}
\Delta_{\infty}u^{\varepsilon}(x) & = & \zeta_{\varepsilon}(x, u^{\varepsilon}) & \mbox{in} & \Omega,\\
u^{\varepsilon}(x) & = & \varphi^{\varepsilon}(x) & \mbox{on} & \partial \Omega,
\end{array}
\right.
\end{equation}
where $\Omega\subset\mathbb{R}^n$ is a bounded domain with a smooth boundary, and $0 \leq \varphi^{\varepsilon} \in C(\overline{\Omega})$ with $\|\varphi^{\varepsilon}\|_{L^{\infty}(\Omega)} \le \mathcal{A}$, for some constant $\mathcal{A}>0$.
The reaction term $\zeta_\varepsilon$ represents the singular perturbation of the model. We are interested in singular behaviors of order $\mbox{O} \left( \frac{1}{\varepsilon} \right)$ along $\e$-level layers $\{u_\varepsilon\sim \varepsilon\}$,  hence we consider (smooth) singular reaction terms $\zeta_\varepsilon \colon \Omega \times \mathbb{R}_{+} \to \mathbb{R}_{+}$ satisfying

\begin{equation}\label{1.2}
	 0 \le \zeta_{\varepsilon}(x, t) \le \frac{\mathcal{B}}{\varepsilon} \chi_{(0,\varepsilon)}(t) + \mathcal{C}, \quad \forall \,\,(x,t) \in \Omega \times \mathbb{R}_{+},
\end{equation}
for some constants $\mathcal{B}, \, \mathcal{C} \ge 0$. Clearly $\zeta_\varepsilon \equiv 0$ satisfies \eqref{1.2}, therefore, to insure that the reaction term is genuinely singular,  we will assume also that
\begin{equation} \label{1.3}
	\mathfrak{R} \defeq \inf\limits_{\Omega \times [a,b]}  \varepsilon \zeta_\varepsilon(x, \varepsilon t)  > 0,
\end{equation}
for some $0\le a<b$, and $\mathfrak{R}$ does not depend on $\varepsilon$.  Heuristically, \eqref{1.3} says that the singular term behaves asymptotically as $\sim \e^{-1} \chi_{(0,\e)}$ plus a nonnegative noise that stays uniformly bounded away from infinity. Singular reaction terms is built up as approximation of unity
\begin{equation} \label{1.4}
	\zeta_{\varepsilon}(x, t) \defeq  \frac{1}{\varepsilon} \beta \left(\frac{t}{\varepsilon}\right) + g_\varepsilon(x),
\end{equation}
are particular (simpler) cases covered by analysis to be developed herein (usually  $\beta$ is a nonnegative smooth real function with $\supp ~\beta = [0,1]$, and $0 \leq c_0 \le g_\varepsilon(x) \le c_1 < \infty$). It is easy to check that the reaction term written in \eqref{1.4} satisfies \eqref{1.2} and \eqref{1.3}.

We were motivated by the study of the following over-determined problem: given $\Omega \subset \mathbb{R}^n$ a domain, functions $0 \leq f, \varphi \in C(\overline{\Omega})$ and $0< g \in C(\overline{\Omega})$, we would like to find a compact ``hyper-surface'' $ \Gamma \defeq \partial \Omega^{\prime} \subset \Omega$ such that the boundary value problem
\begin{equation}  \label{Free}
\left\{
\begin{array}{rcccc}
\Delta_{\infty} u(x)  &=&  f(x) & \mbox{in} & \Omega \backslash \Omega^{\prime}\\
u(x) & = & \varphi(x) & \mbox{on} & \partial \Omega\\
u(x) & = & 0 &\mbox{on} & \Omega^{\prime}\\
\frac{\partial u}{\partial \nu}(x) & = & g(x) & \mbox{on} & \Gamma
\end{array}
\right.
\end{equation}
has a solution. Possible limiting functions coming from \ref{Equation Pe} are natural choices to solve the above problem with $\Gamma=\partial \{u>0\}$ (the free boundary).

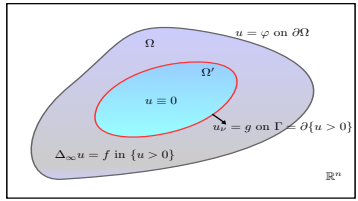
\begin{figure}[h]
\begin{center}

\psscalebox{0.5 0.42} 
{
\begin{pspicture}(0,-4.1)(9.3,4.1)
\definecolor{colour4}{rgb}{0.8,0.8,0.8}
\definecolor{colour3}{rgb}{0.8,0.8,1.0}
\definecolor{colour2}{rgb}{0.4,0.4,0.4}
\definecolor{colour6}{rgb}{0.6,1.0,1.0}
\definecolor{colour5}{rgb}{0.6,0.8,1.0}
\definecolor{colour1}{rgb}{1.0,0.2,0.2}
\psframe[linecolor=black, linewidth=0.04, dimen=outer](9.3,2.1)(0.0,-4.1)
\pscustom[linecolor=black, linewidth=0.04]
{
\newpath
\moveto(8.1,-1.1)
}

\pscustom[linecolor=black, linewidth=0.04]
{
\newpath
\moveto(9.3,-1.5)
}
\pscustom[linecolor=black, linewidth=0.04]
{
\newpath
\moveto(14.9,3.3)
}

\pscustom[linecolor=black, linewidth=0.04]
{
\newpath
\moveto(14.9,2.9)
}
\psbezier[linecolor=colour2, linewidth=0.04, fillstyle=gradient, gradlines=2000, gradbegin=colour3, gradend=colour4](1.6908844,-0.53186655)(3.1618953,0.90770125)(3.113053,1.622456)(5.204675,1.2099068)(7.2962976,0.7973577)(9.114757,-0.26643416)(7.6437454,-1.706002)(6.1727347,-3.1455698)(2.4471457,-3.409379)(1.4782214,-3.4827323)(0.5092973,-3.5560858)(0.21987355,-1.9714344)(1.6908844,-0.53186655)
\psbezier[linecolor=colour1, linewidth=0.04, fillstyle=gradient, gradlines=2000, gradbegin=colour5, gradend=colour6](2.3646572,-1.4787042)(2.3137617,-2.416411)(4.1486554,-2.2604098)(5.1705256,-1.6616218)(6.1923957,-1.0628339)(6.6671977,0.13458377)(5.2729506,0.22547626)(3.8787034,0.31636873)(2.4155526,-0.5409977)(2.3646572,-1.4787042)
\rput[bl](3.7,-1.1){$u \equiv 0$}
\rput[bl](5.2,-0.2){$\Omega^{\prime}$}
\rput[bl](1.3,-2.9){$\Delta_{\infty} u = f \,\,\mbox{in} \,\, \{u>0\}$}
\rput[bl](6.1,1.0){$u = \varphi \,\, \mbox{on} \,\, \partial \Omega$}
\rput[bl](8.5,-3.5){$\mathbb{R}^n$}
\psline[linecolor=black, linewidth=0.04](5.483535,-1.4205831)(5.809861,-1.6918935)
\rput{-17.110401}(0.7442126,1.635026){\psdots[linecolor=black, dotstyle=triangle*, dotsize=0.12820485](5.8063936,-1.692062)}
\rput[bl](5.5,-2.0){$u_{\nu} = g \,\, \mbox{on} \,\, \Gamma = \partial\{u>0\}$}
\rput[bl](3.7,0.7){$\Omega$}
\end{pspicture}
}

\end{center}
\end{figure}

It is important to highlight that, unlike \cite{BdBM} and \cite{MPS}, we can not study \eqref{Equation Pe} as a limit of ``variational solutions'' of the corresponding inhomogeneous problem with $p$-Laplacian on the left hand side of \eqref{Equation Pe}, because several geometric properties and estimates deteriorate, when $p\rightarrow +\infty$, since they depend on $p$ (see, for example, \cite{DPS03, K06, MoWan1}). This indicates the importance of the non-variational approach.

Viscosity solutions of \eqref{Equation Pe} exhibit two ``distinct'' free boundaries: the first one is the set of critical points $\mathcal{C}(u^{\varepsilon}) \defeq \{x \in \Omega \suchthat \nabla u^{\varepsilon}(x) = 0\}$, and the second one is the ``physical'' free boundary, $\Gamma_{\varepsilon} = \{u^{\varepsilon} \thicksim \varepsilon \}$ ($\varepsilon$-level surfaces). We are able to control $u^{\varepsilon}$ in terms of $\dist(x, \Gamma_{\varepsilon})$ and see that these two free boundaries do not intersect.

A problem similar to \eqref{Equation Pe} for a fully nonlinear operators in the left hand side was studied in recent years. In fact, in \cite{RT11} the authors study fully nonlinear uniformly elliptic equations of the form
$$
    F(x, D^2 u^{\varepsilon}) = \zeta_{\varepsilon}(u^{\varepsilon}) \quad \mbox{in} \quad \Omega,
$$
where $\zeta_{\varepsilon} \sim \frac{1}{\varepsilon}\chi_{(0,\varepsilon)}$.
They prove several analytical and geometrical properties of solutions (see also \cite{RS2} for global regularity character and \cite{MoWan1} for an approach with inhomogeneous forcing term). A non-variational setting of the problem was studied in \cite{ART}, where the authors obtain existence and optimal regularity results for the class of fully nonlinear, anisotropic degenerate elliptic problems
$$
|\nabla u^{\varepsilon}|^{\gamma}F(D^2 u^{\varepsilon}) = \zeta_{\varepsilon}(x, u^{\varepsilon}) \quad \mbox{in} \quad \Omega, \,\, \mbox{with} \,\, \gamma\geq 0.
$$
These summarize current results for singularly perturbed non-variational problems.

We also remark that although regularity of infinity harmonic functions is well studied (see \cite{ESav, ES, S05}), regularity results for the inhomogeneous problem $\Delta_{\infty} u=f$ in $\Omega$, are relatively recent and less developed. In this direction it was shown in \cite{Lind} that blow-ups are linear, if $f\in C(\Omega) \cap L^{\infty}(\Omega)$. As a consequence, viscosity solutions of the inhomogeneous problem are Lipschitz continuous and also everywhere differentiable, if $f \in C^1(\Omega) \cap L^{\infty}(\Omega)$. In \cite{BM} Lipschitz regularity was proved for a more general right hand side $f: \Omega \times \R \to \R$ provided $f \in C(\Omega \times \R) \cap L^{\infty}(\Omega \times \R)$.

This paper is organized as follows: in section \ref{DefPreRes} we state some preliminary results, which we use later. In section \ref{Sct Lip} we prove optimal Lipschitz regularity (uniformly in $\varepsilon$). In section \ref{Sct Nondeg} we prove geometric non-degeneracy properties of solutions. As a consequence a Harnack type inequality and porosity of level surfaces are proved. In section \ref{HausEst} we show that for some restricted cases the $(n-1)$-dimensional Hausdorff measure of the free boundary is finite. The corresponding asymptotic limit as $\varepsilon \to 0^{+}$ in \eqref{Equation Pe} is studied in the Section \ref{LimFBP}. We finish the paper analyzing the one-dimensional profile for the limiting free boundary problem in section \ref{r5.2}.

\section{Preliminary results}\label{DefPreRes}

\hspace{0.3cm}We start with the definition of the solution.

\begin{definition}A function $u \in C(\Omega)$ is called a viscosity sub-solution (super-solution) of
$$
 \Delta_{\infty} u = f(x, u(x)) \quad \mbox{in} \quad \Omega,
$$
if whenever $\phi \in C^2(\Omega)$ and $u-\phi$ has a local maximum (minimum) at $x_0 \in \Omega$ there holds
$$
 \Delta_{\infty} \phi(x_0) \geq f(x_0, \phi(x_0)) \quad (\mbox{resp.} \leq f(x_0, \phi(x_0))).
$$
A function $u$ is a viscosity solution when it is a viscosity sub and super-solution at the same time.
\end{definition}
As it was shown in \cite{LW}, the Dirichlet problem
$$
\left\{
\begin{array}{rcccl}
\Delta_{\infty} v(x) & = & f(x) & \mbox{in} & \Omega\\
v(x) & = & g(x) & \mbox{on} & \partial \Omega
\end{array}
\right.
$$
has a unique viscosity solution for $\Omega \subset \R^n$ bounded, provided $g \in C(\partial \Omega)$ and either $\displaystyle \sup_{\Omega} f <0$ or $\displaystyle \inf_{\Omega} f >0$. However, the uniqueness may fail, if $f$ changes the sign (see the counter-example in \cite[Appendix A]{LW}).

We recall a comparison principle result:
\begin{proposition}[{\bf Comparison Principle}, see \cite{BM}, \cite{LW}]\label{CompPrinc} Let $f \in C(\Omega)$ such that $f>0, f<0$ or $f=0$ in $\Omega$. If $u, v \in C(\overline{\Omega})$ satisfy
\begin{equation}\label{eqCompPrinc1}
\Delta_{\infty} u(x) \geq f(x) \geq \Delta_{\infty} v(x) \,\,\, \mbox{in} \,\, \Omega,
\end{equation}
then
\begin{equation}\label{eqCompPrinc1}
\displaystyle \sup_{\Omega}(u-v) = \sup_{\partial \Omega}(u-v).
\end{equation}
\end{proposition}
We construct solutions by Perron's method. We state the following theorem independently of the \eqref{Equation Pe} context, since it may be of independent interest. For the proof we refer to \cite{RT11} (see also \cite{ART}).
\begin{theorem}\label{PerMeth} Let $f \in C^{0, 1}(\Omega \times [0, \infty)) $ be a bounded real function. Suppose that there exist a viscosity sub-solution $\underline{u} \in C(\overline{\Omega}) \cap C^{0, 1}(\Omega)$ and super-solution $\overline{u} \in C(\overline{\Omega}) \cap C^{0, 1}(\Omega)$ to $\Delta_{\infty} u = f(x, u)$ satisfying
$\underline{u} = \overline{u} = \varphi \in C(\partial \Omega)$. Define the class of functions
$$
\mathcal{S}_{\varphi}^{f} \defeq \left\{ w \in C(\overline{\Omega}) \;\middle|\; \begin{array}{c}
w \text{ is a viscosity super-solution to } \\
\Delta_{\infty} u(x) = f(x, u) \text{ in } \Omega \text{ such that } \underline{u} \le w \le \overline{u}\\
\text{ and } w = \varphi \text{ on } \partial \Omega
\end{array}
\right\}.
$$
Then,
\begin{equation}\label{2.1}
	u(x) \defeq \inf_{w \in \mathcal{S}_{\varphi}^{f}} w(x), \,\, \mbox{for} \,\, x \in \overline{\Omega}
\end{equation}
is a continuous viscosity solution to $\Delta_{\infty} u(x) = f(x, u)$ in $\Omega$ with $u=\varphi$ continuously on $\partial \Omega$.
\end{theorem}
Existence of the solution to problem \eqref{Equation Pe} follows by choosing $\underline{u} \defeq \underline{u}^{\varepsilon}$ and $\overline{u} \defeq \overline{u}^{\varepsilon}$ respectively as solutions to the following boundary value problems:
$$
\begin{array}{lll}
\left\{
\begin{array}{rcccl}
\Delta_{\infty} \underline{u}^{\varepsilon} & = &  \sup\limits_{\Omega \times [0, \infty)} \zeta_{\varepsilon} & \mbox{in} & \Omega\\
\underline{u}^{\varepsilon} & = & \varphi^{\varepsilon} & \mbox{on} & \partial \Omega
\end{array}
\right.

& \mbox{and} &

\left\{
\begin{array}{rcccl}
\Delta_{\infty} \overline{u}^{\varepsilon} & = & 0 & \mbox{in} & \Omega\\
\overline{u}^{\varepsilon} & = & \varphi^{\varepsilon} & \mbox{on} & \partial \Omega.
\end{array}
\right.
\end{array}
$$

Then $\underline{u}^\varepsilon \in C(\overline{\Omega})\cap C^{0, 1}(\Omega)$ and $\overline{u}^\varepsilon \in C(\overline{\Omega})\cap C^{0, 1}(\Omega)$ (see \cite{BM}, \cite{Lind} and \cite{LW}) are respectively a viscosity sub and super-solutions of \eqref{Equation Pe}. We state this as a theorem:
\begin{theorem}\label{ExistMinSol}Let $\Omega \subset \R^n$ be a bounded Lipschitz domain and let $\varphi^\varepsilon \in C(\partial \Omega)$ be a nonnegative boundary datum. Then for each fixed $\varepsilon>0$ there exists a (nonnegative) viscosity solution $u^{\varepsilon} \in C(\overline{\Omega})$ to \eqref{Equation Pe}.
\end{theorem}

As a consequence of Proposition \ref{CompPrinc}, we get (uniform) boundness of any family of viscosity solutions.

\begin{lemma}\label{NonnBound}
Let $u^{\varepsilon}$ be a viscosity solution to  \eqref{Equation Pe}. Then there exists a constant $C>0$ independent of $\varepsilon$ such that
$$
    0 \le u^{\varepsilon}(x) \le C \quad \mbox{in} \quad \Omega.
$$
\end{lemma}

Next, we recall (see \cite{RS2}) a Hopf's type lemma below for a future reference.

\begin{lemma}\label{lemma2.1}
Let $u$ be a viscosity solution to
$$
\left\{
\begin{array}{rcl}
	\Delta_{\infty} u = f & \mbox{in} & B_{r}(z)\\
 u \ge 0 & \mbox{in} & B_{r}(z).
\end{array}
\right.
$$
If for some $x_0 \in \partial B_r(z)$,
$$
u(x_0)=0 \quad \textrm{and} \quad \frac{\partial u}{\partial \nu}(x_0) \le \theta,
$$
where $\nu$ is the inward normal vector at $x_0$, then there exists a universal constant
$C >0$ such that
$$
u(z)\le C\theta r.
$$
\end{lemma}

\hspace{0.6cm}\textbf{Notations.} We finish this section by introducing some notations which we shall use in the paper. 
\begin{itemize}
\item[\checkmark] $\Omega_{\varepsilon} \defeq \{x \in \Omega \suchthat 0\leq u^{\varepsilon} \leq \varepsilon\}$ means the $\varepsilon-$level region.
\item[\checkmark] $\Gamma_{\varepsilon} \defeq \{x \in \Omega \suchthat u^{\varepsilon} = \varepsilon\}$ means the $\varepsilon-$level surfaces.
\item[\checkmark]  $\mathfrak{P}(u_0, \Omega^{\prime}) \defeq \{u_0 >0\} \cap \Omega^{\prime}$.
\item[\checkmark] $\mathfrak{F}(u_0, \Omega^{\prime}) \defeq \partial \{u_0 >0\} \cap \Omega^{\prime}$ shall mean the free boundary.
\item[\checkmark] $d_{\varepsilon}(x_0)\defeq \textrm{dist}(x_0, \Omega_{\varepsilon})$.
\item[\checkmark] $\mathcal{N}_\delta (G)\defeq\{x \in \mathbb{R}^n \mid \dist(x,G)<\delta\}$ with $G \subset \R^n$.
\item[\checkmark] $\mathcal{L}^n$ denotes the $n$-dimensional Lebesgue measure.
\item[\checkmark] $\mathcal{H}^{n-1}$ denotes the $(n-1)$-dimensional Hausdorff measure.
\item[\checkmark] $\Omega^{\prime} \Subset \Omega$ means that $\Omega^\prime\subset\overline{\Omega^\prime}\subset\Omega$, and $\overline{\Omega^\prime}$ is compact ($\Omega^\prime$ is compactly contained in $\Omega$).
\item[\checkmark] $\mathfrak{D}(u, B_r(x_0)) \defeq \frac{\mathscr{L}^n(\{u>0\} \cap B_r(x_0))}{\mathscr{L}^n(B_r(x_0))}$ indicates the positive density.
\end{itemize}

\begin{remark} Throughout this paper universal constants are the ones depending only on physical parameters: dimension and structural properties of the problem, i. e. on $n, \mathcal{A}, \mathcal{B}$ and $\mathcal{C}$.
\end{remark}

\section{Uniform Lipschitz regularity} \label{Sct Lip}

\hspace{0.3cm}In this section we prove that viscosity solutions to \eqref{Equation Pe} are (uniformly) locally Lipschitz continuous (which, in view of Theorem \ref{t4.1} below (see also Remark \ref{r6.1}), is optimal).

\begin{theorem}\label{t3.1}
Let $u^{\varepsilon}$ be a viscosity solution to \eqref{Equation Pe}. For every $\Omega^{\prime} \Subset \Omega$, there exists a positive constant $C_0$, independent of $\varepsilon$, such that
$$
	\|\nabla u^{\varepsilon}\|_{L^{\infty}(\Omega^{\prime})} \le C_0(\mathcal{A}, \mathcal{B}, \mathcal{C}, \dist(\Omega^{\prime}, \partial \Omega), \diam(\Omega)).
$$
\end{theorem}

\begin{proof} At first we analyze the closed region $\Omega_{\varepsilon} \defeq \{0\leq u^\varepsilon \leq \varepsilon\}\cap \Omega^{\prime}$. Let $\varepsilon \ll  \frac{1}{3}\dist(\Omega^{\prime}, \partial \Omega)$. We fix $x_0 \in \Omega_{\varepsilon}$ and define $v: B_1 \to \R$ by
$$
v(y)\defeq \frac{u^{\varepsilon}(x_0 + \varepsilon y)}{\varepsilon}.
$$
Then one has
$$
\Delta_{\infty} v=\varepsilon\zeta_{\varepsilon}(x_0 +\varepsilon y,\varepsilon v(y)) \defeq f_\varepsilon(y) \quad \textrm{in} \quad B_1
$$
in the viscosity sense. From \eqref{1.2} we have that
$$
0 \le f_\varepsilon(y)\le \mathcal{B} + \varepsilon\mathcal{C}\le C_\star(\mathcal{B}, \mathcal{C}, \dist(\Omega^{\prime}, \partial \Omega)) .
$$
Since $f_\varepsilon\in C^1$, then $v$ is locally differentiable and moreover (see Theorem 2 and Corollary 2 of \cite{Lind}),
\begin{equation}\label{3.1}
|\nabla v(0)|\le 4\sup_{B_1}v+\frac{1}{2}4^{\frac{1}{3}}\|f_\varepsilon\|_{L^\infty(B_1)}^{\frac{1}{3}}.
\end{equation}
Since
$$
v(0)=\frac{u^{\varepsilon}(x_0)}{\varepsilon}\le 1,
$$
Lemma \ref{NonnBound} and the Harnack inequality (see Theorem 7.1 of \cite{BM}) imply
\begin{equation}\label{3.2}
\|v\|_{L^{\infty}(B_{1})} \le C(\mathcal{A}, \mathcal{B}, \mathcal{C}).
\end{equation}
Combining \eqref{3.1} and \eqref{3.2}, we get
\begin{equation}\label{3.3}
|\nabla u^{\varepsilon}(x_0)|=|\nabla v(0)|\le C_0,
\end{equation}
for some $C_0=C_0(\mathcal{A}, \mathcal{B}, \mathcal{C}, \dist(\Omega^{\prime}, \partial \Omega), \diam(\Omega))>0$ independent of $\varepsilon$.

Now we turn our attention to the case of open region $\{u^{\varepsilon} > \varepsilon\} \cap \Omega'$. Let
$$
\Gamma_{\varepsilon}\defeq \{x \in \Omega' \suchthat  u^{\varepsilon}(x)=\varepsilon\}.
$$
For a fixed $x_1 \in \{u^{\varepsilon}> \varepsilon\} \cap \Omega'$, define $\displaystyle{r \defeq \dist(x_1,\Gamma_{\varepsilon})}$. We define also a function $v_r \colon B_1 \to \mathbb{R}$ by
$$
v_{r}(y)\defeq \frac{u^{\varepsilon}(x_1 + ry)-\varepsilon}{r},
$$
and note that
$$
\Delta_{\infty} v_{r}= r\zeta_{\varepsilon}(x_1 + ry, r v_{r}(y)+\varepsilon) \defeq \mathfrak{g}(y),
$$
in the viscosity sense. The choice of $r$ implies that $u^\varepsilon(x_1 + ry)>\varepsilon$, for every $y\in B_1$, thus, it follows from \eqref{1.2} that $\mathfrak{g}$ is smooth enough and bounded, independently of $\varepsilon$, i.e.,
$$
\|\mathfrak{g}\|_{L^\infty(B_1)}\le K_0(\mathcal{B}, \mathcal{C}, \diam(\Omega)).
$$
Now let $z_0 \in \Gamma_\varepsilon$ be such that $r=|x_1-z_0|$. As in the previous case from \eqref{3.3} one has
\begin{equation}\label{3.4}
|\nabla u^\varepsilon(z_0)|\le C_0(\mathcal{A}, \mathcal{B}, \mathcal{C}, \dist(\Omega^{\prime}, \partial \Omega), \diam(\Omega)).
\end{equation}
Moreover, for $y_0 \defeq \frac{z_0-x_1}{|z_0-x_1|} \in \partial B_1$ we have
$$
v_r(y_0) = 0 \quad \mbox{and} \quad \frac{\partial v_r}{\partial \nu}(y_0) \leq |\nabla v_r(y_0)|\le C_0.
$$
Therefore, by the Lemma \ref{lemma2.1}
$$
v_r(0) \leq C(\mathcal{A}, \mathcal{B}, \mathcal{C}, \dist(\Omega^{\prime}, \partial \Omega), \diam(\Omega)),
$$
and this finishes the proof.
\end{proof}

\section{Further properties of solutions} \label{Sct Nondeg}

\hspace{0.3cm} In this section we prove several properties of solutions. In particular, we show that solutions grow as a linear function out of $\varepsilon$-level surfaces, inside $\{u^\varepsilon > \varepsilon\}$. This is an optimal estimate, when considered uniform in $\e$. The proof is based on building an appropriate barrier function. We consider degenerate elliptic equations of the form
$$
\Delta_{\infty}u = \zeta(x,u) \,\, \text{ in } \,\, \mathbb{R}^n,
$$
where the reaction term satisfies the non-degeneracy assumption:
\begin{equation}\label{4.1}
\inf\limits_{\mathbb{R}^n \times [a,b]}\zeta(x,t)>0.
\end{equation}

\begin{proposition}[{\bf Infinity Laplacian's Barrier}]\label{p4.1} Let $0<a<b<1$ be fixed. For $\alpha$ and $A_0$ positive numbers (to be chosen) {\it a posteriori}, there exists a radially symmetric function $\Theta_L \colon \mathbb{R}^n \rightarrow \mathbb{R}$ satisfying
\begin{enumerate}
\item[\checkmark] $\Theta_L \in W^{2,\infty}(\mathbb{R}^n)\cap C^{1,1}_{\loc}(\mathbb{R}^n)$,
\item[\checkmark] \begin{equation}\label{4.4}
\Delta_{\infty}\Theta_{L}(x)\le\zeta(x,\Theta_L(x)) \quad \mbox{in}\quad\mathbb{R}^{\,n},
\end{equation}
\item[\checkmark] there exists a universal $\kappa_0>0$ constant such that
\begin{equation}\label{4.5}
\Theta_L(x) \geq 4\kappa_0L \quad \mbox{for} \quad |x| \geq 4L,
\end{equation}
where $L \geq L_0 \defeq \sqrt{\frac{b-a}{A_0}}$.
\end{enumerate}
\end{proposition}
\begin{proof} Define
\begin{equation}\label{4.2}
\begin{array}{cc}
\Theta_L(x)\defeq &
\left\{
\begin{array}{ccl}
a & \mbox{for} & 0\leq |x| < L; \\
A_0\left(|x| - L\right)^2+a & \mbox{for} & L\leq |x| < L+L_0; \\
\psi(L)-\phi(L)|x|^{-\alpha} & \mbox{for} & |x| \geq L+L_0.\\
\end{array}
\right.
\\
\end{array}
\end{equation}
where
\begin{equation}\label{4.3}
\phi(L)= \frac{2}{\alpha}\sqrt{(b-a)A_0}\left(L+L_0 \right)^{1+\alpha}
\,\, \mbox{and} \,\,
\psi(L)= b+\phi(L)\left(L+L_0\right)^{-\alpha},
\end{equation}
Clearly $\Theta_L \in W^{2,\infty}(\mathbb{R}^n)\cap C^{1,1}_{\textrm{loc}}(\mathbb{R}^n)$. Moreover, for $0\leq|x|<L$ the inequality \eqref{4.4} is true. In the region $L \leq|x|< L+L_0$, we have
$$
D_{i} \Theta_{L}(x)=2A_0\frac{(|x|-L)}{|x|}x_i
$$
and
$$
D_{ij} \Theta_L(x) = 2A_0\left[\left(\frac{1}{|x|^2}- \frac{(|x|-L)}{|x|^3}\right) x_i\cdot x_j+ \frac{(|x|-L)}{|x|}\delta_{ij} \right].
$$
Therefore, we obtain
\begin{eqnarray*}
	\Delta_{\infty} \Theta_{L}(x) & = & \sum_{i,j=1}^{n} D_{i} \Theta_L \cdot D_{j} \Theta_{L} \cdot D_{ij} \Theta_L \\
	&=& 8A^{3}_{0} \frac{(|x|-L)^{2}}{|x|^2} \sum_{i,j=1}^{n} \left[\left(\frac{1}{|x|^2} - \frac{(|x|-L)}{|x|}\right) x^{2}_{i}x^{2}_{j} + \frac{|x|-L}{|x|} x_i\cdot x_j \delta_{ij}\right]\\
	&=& 8A^{3}_{0} \frac{(|x|-L)^2}{|x|^2} \left[\left(\frac{1}{|x|^2} - \frac{(|x|-L)}{|x|}\right) |x|^4 + \frac{(|x|-L)}{|x|} |x|^2\right]\\
	&=& 8A^{3}_{0}   \frac{(|x|-L)^2}{|x|^2}|x|^2 = 8A^{3}_{0} (|x|-L)^3  \le 8A^{3}_{0}L^3_0\\
	&=& (2\sqrt{A_0(b-a)})^3.
\end{eqnarray*}
By construction
$$
   a \le \Theta_L(x) \le b
$$
and so, for $A_0$ sufficiently small, we get
\begin{equation}\nonumber
\begin{array}{c}
\Delta_{\infty} \Theta_{L}(x) \leq \inf\limits_{\mathbb{R}^n \times [a,b]} \zeta(x,t)\leq \zeta(x,\Theta_L(x)). \\
\end{array}
\end{equation}
Now, let us turn our attention to the set $|x| \geq L+L_0$. Direct computation shows that
$$
D_{i}\Theta_{L}(x) = \alpha \phi(L) \frac{x_i}{|x|^{\alpha +2}}
$$
and
$$
D_{ij} \Theta_L(x) =\alpha\phi(L)|x|^{-(\alpha+2)}\left( -\frac{(\alpha+2)}{|x|^2} x_i x_j+ \delta_{ij}\right),
$$
hence
\begin{eqnarray*}
\Delta_{\infty} \Theta_{L}(x) = -\alpha^{3} \phi^3(L)(\alpha+1) \frac{1}{|x|^{3\alpha +4}}.
\end{eqnarray*}
Finally, for $\alpha>0$ we get
$$
\Delta_{\infty} \Theta_{L}(x) \le 0 \le \zeta(x,\Theta_L(x)).
$$
Therefore, $\Theta_L$ satisfies \eqref{4.4}. Finally, by \eqref{4.3}
$$
|x| \geq 4L \geq 2(L+L_0) = 2\, \left( \dfrac{\phi(L)}{\psi(L)-b}\right)^{\frac{1}{\alpha}}
$$
and hence
$$
\Theta_L(x) = \psi(L)-\phi(L)|x|^{-\alpha} \geq \psi(L) - 2^{-\alpha}(\psi(L)-b) \geq C_\alpha\psi(L),
$$
for $\alpha > 1$. Therefore,
$$
\Theta_L(x) \geq 4\kappa_0L,
$$
where $\kappa_0>0$ depends on $n$ and $(b-a)$.
\end{proof}

\subsection{Linear growth}

\hspace{0.3cm} In order to establish lower bounds on the growth speed of the solution to \eqref{Equation Pe} inside the set $\{u^\varepsilon > \varepsilon \}$, the strategy now is to consider appropriate scaling versions of the universal barrier $\Theta_L$.
\begin{theorem}\label{t4.1}
Let $u^{\varepsilon}$ be a solution of \eqref{Equation Pe}. There exists  a universal constant $c>0$ such that for any $x_0 \in \{u^{\varepsilon} > \varepsilon\}$ and $0<\varepsilon\leq d_\varepsilon(x_0) \ll 1$ one has
$$
 u^{\varepsilon}(x_0) \ge cd_{\varepsilon}(x_0).
$$
\end{theorem}

\begin{proof}
Without loss of generality we assume that $x_0 = 0$. Set $\eta=\dfrac{d_\varepsilon(0)}{3}$ and define
$$
\Theta_\varepsilon(x)\defeq \varepsilon\Theta_{\frac{\eta}{4\varepsilon}} \left (\frac{x}{\varepsilon} \right ).
$$
Using\eqref{4.5} and \eqref{4.2} we verify that for $4L_0 \varepsilon\leq\eta$,
\begin{equation}\label{4.6}
\Theta_\varepsilon(0) = a\varepsilon \quad \mbox{and} \quad 	{\Theta_\varepsilon}_{ \big | \partial B_{\eta}} \geq \kappa_0 \eta.
\end{equation}
Now, we claim that there exists a $z_0\in\partial B_{\eta}$ such that
\begin{equation}\label{4.7}
\Theta_{\varepsilon}(z_0) \le u^{\varepsilon}(z_0).
\end{equation}
In fact, if
$$
\Theta_{\varepsilon}(x) > u^{\varepsilon}(x) \quad \mbox{in} \quad \partial B_{\eta},
$$
then the auxiliary function
$$
v^{\varepsilon}\defeq \min\{\Theta_{\varepsilon},u^{\varepsilon}\}
$$
would be a super-solution to \eqref{Equation Pe}, but $v^\varepsilon$ is strictly below $u^{\varepsilon}$, which contradicts the minimality of $u^\varepsilon$.
Therefore, by \eqref{4.6} and \eqref{4.7}, we obtain
\begin{equation}\label{4.8}
     \kappa_0\eta \le  \Theta_{\varepsilon}(z_0) \leq u^{\varepsilon}(z_0)\le\sup\limits_{B_\eta}u^\varepsilon.
\end{equation}
Furthermore, $u^\varepsilon$ satisfies (in the viscosity sense)
$$
   c_0 \leq \Delta_{\infty} u^{\varepsilon} \leq c_1 \quad \mbox{in} \quad B_{3\eta}.
$$
Hence, by Harnack inequality (see Theorem 7.1 of \cite{BM}), we get
$$
\sup\limits_{B_{\eta}} u^\varepsilon \leq 9u^\varepsilon(0)+ 12\sigma\left(\left(\frac{3\eta}{2}\right)^4c_1 \right)^{1/3}.
$$
Thus, by \eqref{4.8}
$$
   u^\varepsilon(0) \geq \frac{1}{9}\left(\kappa_0 - C \eta^{1/3} \right)\eta.
$$
Finally, by taking $\eta>0$ small enough we conclude
$$
  u^\varepsilon(0) \geq c \,\eta.
$$
for some $0<c<1$ (independent of $\varepsilon$).
\end{proof}

As a consequence of the Lipschitz regularity, Theorem \ref{t3.1} and Theorem \ref{t4.1}, we are able to completely control $u^{\varepsilon}$ in terms of $d_{\varepsilon}(x_0)$.

\begin{corollary}\label{c4.1}
For a sub-domain $\Omega^{\prime} \Subset \Omega$, there exists $C>0$, depending on universal parameters and $\Omega^{\prime}$, such that for $x_0 \in \mathfrak{P}(u^{\varepsilon} -\varepsilon, \Omega^{\prime})$ and $\varepsilon\leq d_{\varepsilon}(x)$, there holds
$$
C^{-1} d_{\varepsilon}(x_0) \le u^{\varepsilon}(x_0) \le C\, d_{\varepsilon}(x_0).
$$
\end{corollary}
\begin{proof}
The inequality from below is exactly the Theorem \ref{t4.1}. Now take $y_0 \in \mathfrak{F}(u^{\varepsilon} - \varepsilon, \Omega^{\prime})$, such that $|y_0-x_0|=d_{\varepsilon}(x_0)$. From Theorem \ref{t3.1},
$$
u^{\varepsilon}(x_0) \le C\,d_{\varepsilon}(x_0) + u^\varepsilon(y_0) \le C \,\, d_{\varepsilon}(x_0),
$$
and the corollary is proved.
\end{proof}

\subsection{Strong non-degeneracy}

\hspace{0.3cm} Next we see that solutions are strongly non-degenerate close to $\varepsilon$-level sets. This means that the maximum of $u^{\varepsilon}$ on the boundary of a ball $B_r$ centered in $\{u^{\varepsilon} > \varepsilon\}$ is of order $r$.

\begin{theorem}\label{t4.2}
Let $\Omega^{\prime} \Subset \Omega$. There exists a universal constant $c >0$ such that for $x_0 \in \mathfrak{P}(u^{\varepsilon} -\varepsilon, \Omega^{\prime})$, $\varepsilon\leq\rho \ll 1$,  there holds
$$
c\, \rho < \sup_{B_{\rho}(x_0)}u^{\varepsilon} \le c^{-1} ( \rho + u^{\varepsilon}(x_0)).
$$
\end{theorem}
\begin{proof}
By taking $\Theta_\e(x)=\e\Theta_{\frac{\rho}{4\e}}(x)$ we have
$$
u^{\varepsilon}(z)>\Theta_{\varepsilon}(z),
$$
for some point $z\in\partial B_{\rho}(x_0)$. Note that
$$
\kappa_0\rho\le\Theta_{\varepsilon}(z)< u^{\varepsilon}(z)\le\sup_{B_\rho(x_0)} u^{\varepsilon},
$$
where $\kappa_0$ is as in Proposition \ref{p4.1}.
The upper estimate is a direct consequence of the Lipschitz regularity.
\end{proof}
As a consequence we get a positive density result.
\begin{corollary}\label{c4.2}
Let $x_0\in \{u^\varepsilon>\varepsilon\}$ and $\varepsilon\leq \rho \ll 1$. There exists a universal constant $c_0\in(0,1)$ such that
$$
\mathfrak{D}(u^{\varepsilon}-\varepsilon, B_{\rho}(x_0))\geq c_0.
$$
\end{corollary}
\begin{proof}
As we saw in the previous theorem, there exists $y_0\in  B_{\rho}(x_0)$ such that
$$
u^{\varepsilon}(y_0)\ge c_0 \rho.
$$
On the other hand, by Lipschitz regularity, for $z\in B_{\kappa \rho}(y_0)$, we have
$$
u^{\varepsilon}(z)+C\kappa\rho\ge u^{\varepsilon}(y_0).
$$
Thus, by using the estimates from above, we are able to choose $\kappa>0$ small enough in order to have
$$
z\in B_{\kappa \rho}(y_0) \cap B_{\rho}(x_0) \quad \textrm{and} \quad u^{\varepsilon}(z)>\varepsilon.
$$
So we conclude that there exists a portion of $B_\rho(x_0)$ with volume of order $\sim\rho^n$ within $\{u^\varepsilon>\varepsilon\}$. Therefore, we have a uniform positive density result for the solution of \eqref{Equation Pe}. More precisely,
$$
\mathscr{L}^n(B_{\rho}(x_0) \cap \{u^{\varepsilon} > \varepsilon\}) \ge \mathscr{L}^n ( B_{\rho}(x_0) \cap B_{\kappa \rho}(y_0))=c_0 \,\mathscr{L}^{n}(B_{\rho}(x_0)),
$$
for some constant  $c_0>0$ independent of $\varepsilon$.
\end{proof}

\subsection{Harnack type inequality}

\hspace{0.3cm} For solutions of \eqref{Equation Pe} the Harnack inequality is valid for balls that touch the free boundary along the $\varepsilon$-layers, i.e., $\partial\{u^{\varepsilon} > \varepsilon\}$.

\begin{theorem}\label{ThmHarIneq} Let $u^\varepsilon$ be a solution of \eqref{Equation Pe}. Let also $x_0 \in \{u^{\varepsilon} > \varepsilon\}$ and $\varepsilon\leq d \defeq d_{\varepsilon}(x_0)$. Then,
$$
\displaystyle \sup_{B_{\frac{d}{2}}(x_0)} u^{\varepsilon}(x) \leq C \inf_{B_{\frac{d}{2}}(x_0)} u^{\varepsilon}(x)
$$
for a universal constant $C>0$ independent of $\varepsilon$.
\end{theorem}

\begin{proof} Let $z_1, z_2$ be extremal points for $u^{\varepsilon}$ in $\overline{B_{\frac{d}{2}}(x_0)}$, i.e.,
$$
\displaystyle \inf_{B_{\frac{d}{2}}(x_0)} u^{\varepsilon}(x) = u^{\varepsilon}(z_1) \quad \mbox{and} \quad \sup_{B_{\frac{d}{2}}(x_0)} u^{\varepsilon}(x) = u^{\varepsilon}(z_2).
$$
Since $d_{\varepsilon}(z_1) \geq \frac{d}{2}$, by Corollary \ref{c4.1}
\begin{equation}\label{eqHar6.1}
u^{\varepsilon}(z_1) \geq C_1d.
\end{equation}
Moreover, by Theorem \ref{t4.2}
\begin{equation}\label{eqHar6.2}
u^{\varepsilon}(z_2) \leq C_2\left(\frac{d}{2} + u^{\varepsilon}(x_0)\right).
\end{equation}
Taking $y \in \partial \{u^{\varepsilon} > \varepsilon\}$ such that $d=|x_0-y|$ and $z\in\overline{B_d(y)}\cap\{u^\varepsilon>\varepsilon\}$, we get from Corollary \ref{c4.1} and Theorem \ref{t4.2}
\begin{equation}\label{eqHar6.3}
u^{\varepsilon}(x_0) \leq \sup\limits_{B_d(z)}u^\varepsilon \leq C_2(d+ u^{\varepsilon}(z)) \leq C_3d.
\end{equation}
Combining \eqref{eqHar6.1}, \eqref{eqHar6.2} and \eqref{eqHar6.3}, we conclude
$$
\displaystyle \sup_{B_{\frac{d}{2}}(x_0)} u^{\varepsilon}(x) \leq C \inf_{B_{\frac{d}{2}}(x_0)} u^{\varepsilon}(x).
$$
\end{proof}

\subsection{Porosity of the level surfaces}\label{porous}

\hspace{0.3cm} As a consequence of the growth rate and the non-degeneracy property, we get porosity of level sets.

\begin{definition}\label{d5.1}
A set $E\subset\mathbb{R}^n$ is called porous with porosity $\delta>0$, if $\exists\,R>0$ such that
$$
   \forall x\in E, \,\,\,\forall r\in(0,R),\,\,\,\exists y\in\mathbb{R}^n\,\textrm{ such that }\,B_{\delta r}(y)\subset B_r(x)\setminus E.
$$
\end{definition}
A porous set of porosity $\delta$ has Hausdorff dimension not exceeding $n-c\delta^n$, where $c=c(n)>0$ is a constant depending only on $n$. In particular, a porous set has Lebesgue measure zero (see, for example, \cite{Z88}).

\begin{theorem}\label{t5.2}
Let $u^\varepsilon$ be a solution of \eqref{Equation Pe}. Then the level sets $\partial \{u^{\varepsilon} >\varepsilon\}$ are porous with porosity constant independent of $\varepsilon$.
\end{theorem}
\begin{proof}
Let $R>0$ and $x_0\in\Omega$ be such that $\overline{B_{4R}(x_0)}\subset\Omega$.

We aim to prove the set $\mathfrak{F}(u^{\varepsilon}-\varepsilon, B_R(x_0))$ is porous.

Let $x\in \mathfrak{F}(u^{\varepsilon}-\varepsilon, B_R(x_0))$. For each $r\in(0,R)$ we have $\overline{B_r(x)}\subset B_{2R}(x_0)\subset\Omega$. Let $y\in\partial B_r(x)$ such that $u^{\varepsilon}(y)=\sup\limits_{\partial B_r(x)}u^{\varepsilon}$. By non-degeneracy
\begin{equation}\label{5.1}
    u^{\varepsilon}(y)\geq cr,
\end{equation}
where $c>0$ is a constant. On the other hand, we know that near the free boundary
\begin{equation}\label{5.2}
    u^{\varepsilon}(y)\leq Cd_{\varepsilon}(y),
\end{equation}
where $C>0$ is a constant, and $d_{\varepsilon}(y)$ is the distance of $y$ from the set $\overline{B_{2R}(x_0)}\cap\Gamma_{\varepsilon}$. Now, from \eqref{5.1} and \eqref{5.2} we get
\begin{equation}\label{5.3}
    d_{\varepsilon}(y)\geq\delta r
\end{equation}
for a positive constant $\delta<1$.

Let now $y^*\in[x,y]$ be such that $|y-y^*|=\frac{\delta r}{2}$, then it is not hard to see that
\begin{equation}\label{5.4}
   B_{\frac{\delta}{2}r}(y^*)\subset B_{\delta r}(y)\cap B_r(x).
\end{equation}
Indeed, for each $z\in B_{\frac{\delta}{2}r}(y^*)$
$$
   |z-y|\leq |z-y^*|+|y-y^*|<\frac{\delta r}{2}+\frac{\delta r}{2}=\delta r,
$$
and
$$
   |z-x|\leq|z-y^*|+\big(|x-y|-|y^*-y|\big)<\frac{\delta r}{2}+\left(r-\frac{\delta r}{2}\right)=r,
$$
and \eqref{5.4} follows.

Since by \eqref{5.3} $B_{\delta r}(y)\subset B_{d_{\varepsilon}(y)}(y)\subset\{u^{\varepsilon}>\varepsilon\}$, then
$$
   B_{\delta r}(y)\cap B_r(x)\subset\{u^{\varepsilon}>\varepsilon\},
$$
which together with \eqref{5.4} provides
$$
   B_{\frac{\delta}{2}r}(y^*)\subset B_{\delta r}(y)\cap B_r(x)\subset B_r(x)\setminus\partial\{u_{\varepsilon}>\varepsilon\}\subset B_r(x)\setminus \mathfrak{F}(u^{\varepsilon}-\varepsilon, B_R(x_0)).
$$
\end{proof}

\section{Hausdorff measure estimates}\label{HausEst}

\hspace{0.3cm}In this section we prove the finiteness of the $(n-1)$-dimensional Hausdorff measure of level surfaces. For that we restrict ourselves to the case when the reaction term, which propagates up to the free boundary, is non-degenerate. Suppose that $a=0$ in \eqref{1.3} and for some $b>0$
\begin{equation}\label{eq1Hausd}
\displaystyle \mathfrak{R}_0 \defeq \inf_{\Omega \times [0, b]} \varepsilon \zeta_{\varepsilon}(x, \varepsilon t)>0.
\end{equation}

\begin{definition}[{\bf Asymptotic Concavity Property}]\label{defACP} We say that an operator $F: \Omega \times Sym(n) \to \R$ is \textit{asymptotically concave}, if there exists
$$
\mathfrak{A} \in \mathcal{A}_{\lambda,\Lambda}\defeq \left\{A \in \textrm{Sym}(n) \suchthat \lambda \|\xi\|^2 \le \sum\limits_{i,j=1}^{n}A_{ij} \xi_i \xi_j \le \Lambda \|\xi\|^{2}, \, \forall \, \xi \in \mathbb{R}^n\right\}
$$
and a continuous function $\omega_F: \Omega \times Sym(n) \to \R$ such that
\begin{equation}\label{ACP} \tag{{\bf ACP}}
    F(x, M) \leq \tr(\mathfrak{A}(x) \cdot M) + \omega_F(x, M), \, \forall \,\,(x, M)\in \Omega \times Sym(n),
\end{equation}
with
\begin{equation}\label{limcond}
  \displaystyle \lim_{\|M\| \to \infty} |\omega_F(x, M)| \defeq \mathcal{K} < \infty, \quad \forall \,\, x\in \Omega.
\end{equation}
\end{definition}

\begin{remark} The \eqref{ACP} condition is weaker than concavity assumption. Geometrically, it means that for each $x \in \Omega$ fixed, there exists a hyperplane which decomposes $\R \times Sym(n)$ in two semi-spaces such that the graph of $F(x, \cdot)$ is always below this hyperplane. Moreover, by assuming $F(x, 0) = 0$, the assumption \eqref{limcond} means that the distance from the hyperplane to the graph of $F$ goes to infinity for matrices with big enough norms (see \cite{ART} and \cite{MoWan2}).
\end{remark}

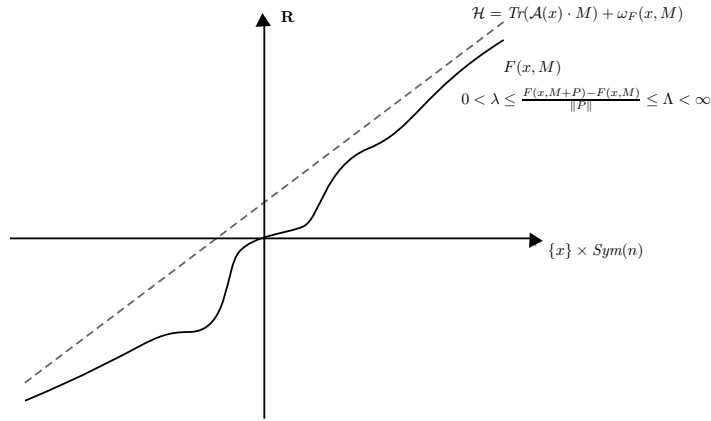
\begin{figure}[h]
\begin{center}
\psscalebox{0.6 0.6} 
{
\begin{pspicture}(0,-4.575)(15.51,4.575)
\definecolor{colour0}{rgb}{0.4,0.4,0.4}
\psbezier[linecolor=black, linewidth=0.04](0.33142856,-4.175)(3.0586123,-3.0321429)(3.0586123,-2.6511905)(3.9676735,-2.6511905)(4.8767347,-2.6511905)(4.7518973,-1.2342604)(5.082591,-0.86615646)(5.4132843,-0.49805245)(6.110321,-0.4755564)(6.490648,-0.31468254)(6.8709745,-0.15380867)(6.96,1.025)(7.9542856,1.425)(8.948571,1.825)(9.11902,2.682143)(10.937143,3.825)
\psline[linecolor=colour0, linewidth=0.04, linestyle=dashed, dash=0.17638889cm 0.10583334cm](0.33142856,-3.775)(10.937143,4.225)
\psline[linecolor=black, linewidth=0.04](5.634286,4.225)(5.634286,-0.575)(11.6,-0.575)
\psline[linecolor=black, linewidth=0.04](5.634286,-4.575)(5.634286,-0.575)(0.0,-0.575)
\rput[bl](11.916667,-1.0321429){$\{x\} \times\textit{Sym}(n)$}
\rput[bl](10.937143,3.025){$F(x, M)$}
\rput[bl](10.233334,4.225){$\mathcal{H} = \textit{Tr}(\mathcal{A}(x) \cdot M)+ \omega_F(x, M)$}
\rput[bl](6.0,4.225){$\mathbf{R}$}
\rput[bl](10.0,2.225){$0<\lambda\leq \frac{F(x, M+P)-F(x, M)}{\|P\|}\leq \Lambda < \infty$}
\rput{3.3842099}(0.2641535,-0.32306013){\psdots[linecolor=black, dotstyle=triangle*, dotsize=0.3](5.6,4.225)}
\rput{-208.64128}(21.506895,-6.980531){\psdots[linecolor=black, dotstyle=triangle*, dotsize=0.027228916](11.644444,-0.75277776)}
\rput{-92.6068}(12.84383,11.024456){\psdots[linecolor=black, dotstyle=triangle*, dotsize=0.3](11.688889,-0.7083333)}
\end{pspicture}
}
\end{center}
\end{figure}

\begin{definition}\label{defclaS} Let $v$ be the solution of \eqref{Equation Pe}. We write $v\in\mathcal{S}(F, G, H)$, if
$$
\Delta_\infty v \leq G(|Dv|)F(x, D^2v) + H(x, |Dv|),
$$
where
\begin{itemize}
\item[\checkmark] $F: \Omega \times Sym(n) \to \R$ is a fully nonlinear uniformly elliptic operator with $F(x, 0)=0$;
\item[\checkmark] $G: \R_+ \to \R$ is a non-negative continuous function and injective;
\item[\checkmark] $H: \Omega \times \R_{+} \to \R$ is a bounded continuous function.
\end{itemize}
\end{definition}

\begin{example}[{\bf $\varphi$-Laplacian operator}]\label{Example}
The $\varphi$-Laplacian operator in Orlicz-Sobolev spaces can be defined as
$$
\Delta_{\varphi} u =  \frac{\varphi(|\nabla u|)}{|\nabla u|}\left[ \Delta u + \left\{ \frac{\varphi^{\prime}(|\nabla u|)|\nabla u|}{\varphi(|\nabla u|)} - 1\right\}\frac{\Delta_{\infty} u}{|\nabla u|^2} \right].
$$
for an appropriate increasing function $\varphi: [0, \infty) \to [0, \infty)$ satisfying the generalized Ladyzhenskaya-Ural'tseva condition:
$$
0<g_0 \leq \frac{\varphi^{\prime}(t)t}{\varphi(t)} \leq g_1, \quad \mbox{if} \quad t>0,
$$
where $g_0$ and $g_1$ are constants. Therefore, for a $\varphi-$harmonic function one has (where $\nabla u \neq0$)

$$
\Delta_{\infty} u
\leq \frac{\varphi(|\nabla u|)|\nabla u|^2}{\varphi^{\prime}(|\nabla u|)|\nabla u|-\varphi(|\nabla u|)} \Delta u.
$$

\end{example}

\begin{example}[{\bf Convex functions}] For convex functions we have following relation
$$
\Delta_{\infty} u  = \langle D^2 u Du, Du\rangle \leq |\nabla u|^2 \Delta u,
$$
since $\| D^2 u\|$ is controlled by $\Delta u$.
\end{example}

The proof of the following proposition is similar to the corresponding result from \cite{ART}. We sketch it here for reader's convenience.

\begin{proposition}\label{ac1} For the every fixed $\Omega^{\prime} \Subset \Omega$, $\rho < \dist(\Omega^{\prime}, \partial \Omega)$ and $C \gg 1$, there exists a universal $\varepsilon_0$ such that
\begin{equation}\label{eq2Hausd}
\displaystyle \int_{B_{\rho}(x_{\varepsilon})} [\zeta_{\varepsilon}(x, u^{\varepsilon}(x))-C]\,dx \geq 0,
\end{equation}
for any $x_{\varepsilon} \in \mathfrak{F}(u^{\varepsilon}- \varepsilon,  \Omega^{\prime})$
whenever $\varepsilon \leq \varepsilon_0$.
\end{proposition}

\begin{proof}
If \eqref{eq2Hausd} is not true, then there are $C_0>0$ and $\rho<  \dist(\Omega^{\prime},\partial\Omega)$ such that
\begin{equation*}
\int_{B_{\rho}(x_k)}\left(\zeta_{\varepsilon_{k}}(x, u^{\varepsilon_k})-C_0\right)\,dx<0,
\end{equation*}
for points $x_{\varepsilon_k} \in \mathfrak{F}(u^{\varepsilon_k}-\varepsilon_k ,\Omega')$ and a sequence $\varepsilon_{k}\to 0$ as $k\to \infty$. Define
\[
v_k(y) \defeq \frac{bu^{\varepsilon_k}(x_{\varepsilon_k}+\varepsilon_k\,y)}{\varepsilon_k}.
\]
Then
\begin{equation}\label{ac5}
\int_{B_{\rho/\varepsilon_k}}\left((\varepsilon_k b^{-1})\zeta_{\varepsilon_k} (x_{{\varepsilon_k}} + {\varepsilon_k} y, \varepsilon_k b^{-1}v_k)-C_0\varepsilon_k b^{-1}\right)\,dx<0.
\end{equation}
Note that
$$
\| \Delta_{\infty} v_{k} \|_{L^\infty(B_{\rho/\varepsilon_k})} \le \frac{\mathcal{B} + \mathcal{C}}{b},
$$
independent of $\varepsilon$.

By the regularity of $v_k$ one has (up to a subsequence) that
\[
	 v_\infty \defeq \lim_{k\to \infty} v_k ,
\]
in the $C_\text{loc}^{0, \alpha}$ topology.  Combining  \eqref{eq1Hausd} and \eqref{ac5}, we deduce that
\begin{equation*}
\text{either } v_\infty \equiv 0, \quad    \mbox{ or else } \quad v_\infty \geq b,   \text{ everywhere in  } \mathbb{R}^n.
\end{equation*}
The first case is not possible since $v_\infty(0)=b>0$. If $v_\infty\geq b$, we have that $0$ is a minimum point, which leads to a contradiction, since by non-degeneracy
\[
0=|\nabla v_\infty(0)|=|\nabla u^{\varepsilon_k}(0)|+ \text{o}(1)\geq c  >0.
\]
\end{proof}

Thus, combining the \eqref{ACP} condition and the Proposition \ref{ac1}, we obtain:
\begin{lemma}\label{l5.1}
Let $u^\varepsilon\in\mathcal{S}(F, G, H)$  with $F$ being asymptotically concave and let $x_{\varepsilon} \in \mathfrak{F}(u^{\varepsilon}- \varepsilon,  \Omega^{\prime})$. Then
\begin{equation}\label{ac3}
\int_{B_\rho(x_\varepsilon)}A_{ij}\,u^\varepsilon_{ij}\,dx\geq0.
\end{equation}
\end{lemma}
\begin{proof}
Note that
$$
F(x,D^2 u^{\varepsilon}) \ge [\zeta_{\varepsilon}(x, u^{\varepsilon})-H(x,|\nabla u^{\varepsilon}|)]G(|\nabla u^{\varepsilon}|)^{-1}
$$
in $\{u^{\varepsilon} > \varepsilon\} \cap \Omega^{\prime}$, for any $\Omega^{\prime} \Subset \Omega$. Hence,  by Lipschitz regularity and properties of $G$ and $H$, one has
 $$
 F(x,D^2 u^{\varepsilon}) \ge [\zeta_{\varepsilon}(x, u^{\varepsilon})-C_H]G(C)^{-1}.
$$
Therefore, by \eqref{ACP} condition
\begin{eqnarray*}
\int_{B_\rho(x_\varepsilon)}A_{ij}\,u^\varepsilon_{ij}\,dx &\geq& \int_{B_\rho(x_\varepsilon)} \left[(\zeta_\varepsilon(x, u^\varepsilon)-C_H)G(C)^{-1} - \mathcal{K}\right]\,dx\\
&\ge& G(C)^{-1}\int_{B_\rho(x_\varepsilon)} \left[\zeta_\varepsilon(u^\varepsilon)-(C_H+G(C)\mathcal{K}) \right]\,dx,
\end{eqnarray*}
where $C>0$ comes from the  universal control on the Lipschitz norm in $B_{\rho}(x_\varepsilon)$. Combining the estimate above and the Proposition \ref{ac1}, we obtain \eqref{ac3}.

\end{proof}

Lemma \ref{l5.1} plays a crucial role in the study of regularity of level surfaces, since it leads to the following result (see Theorem 5.6 in \cite{ART}):

\begin{theorem}\label{Hausdf}
Let $\Omega^{\prime} \Subset \Omega$ and $u^\varepsilon\in\mathcal{S}(F, G, H)$ with $F$ being asymptotically concave. There exists a  $C>0$ constant depending on $\Omega^{\prime}$ such that
\begin{equation}\label{eq10}
	\mathcal{H}^{n-1}(\mathfrak{P}\left( u^\varepsilon-C_1\varepsilon, B_\rho(x_\varepsilon)\right)) \le C \rho^{n-1},
\end{equation}
for some $C_1>1$ and for all $x_\varepsilon \in \mathfrak{F}\left( u^\varepsilon-C_1\varepsilon, \Omega^{\prime}\right)$, provided $d_{\varepsilon}(x_{\varepsilon})<\textrm{dist}(\Omega', \partial \Omega)$ and $C_1 \varepsilon\leq\rho$.
\end{theorem}

\section{The limiting problem}\label{LimFBP}

\hspace{0.3cm} As a consequence of Theorem \ref{t3.1} and Lemma \ref{NonnBound} we obtain the following result:
\begin{theorem} If $\{u^{\varepsilon}\}_{\varepsilon>0}$ is a solution to \eqref{Equation Pe}, then for any sequence $\varepsilon_k \to 0^{+}$ there exist a subsequence $\varepsilon_{k_j} \to 0^{+}$ and $u_0\in C^{0,1}_{\loc}(\Omega)$ such that

\begin{enumerate}
	\item[(1)] $u^{\varepsilon_{k_j}} \to u_0$ locally uniformly in $\Omega$;
	\item[(2)] $0 \le u_0(x) \le K_0$ in $\overline{\Omega}$ for some constant $K_0$ independent of $\varepsilon$;
	\item[(3)] $ \Delta_{\infty} u_0(x)  = g(x)$ in $\Omega \backslash \mathfrak{F}(u_0, \Omega^{\prime})$, with $g$ being a bounded and nonnegative continuous function.
\end{enumerate}
\end{theorem}
\begin{remark}\label{r6.1}It follows from $(3)$ (using the corresponding  regularity result from \cite{Lind}) that $u_0$ is locally differentiable in $\mathfrak{P}(u_0, \Omega^{\prime})$. However, that property deteriorates as $\dist(\partial \Omega^{\prime}, \partial \{u_0 > 0\}) \to 0$. On the other hand, the gradient remains controlled even when $\dist(x_0, \mathfrak{F}(u_0, \Omega^{\prime})) \to 0$.
\end{remark}
Hereafter we will use the following definition when referring to $u_0$:
$$
u_0(x) \defeq \lim\limits_{j \to \infty}u^{\varepsilon_j}(x).
$$
\begin{theorem}\label{limite} Let $\Omega^{\prime} \Subset \Omega$. Fix $x_0 \in \mathfrak{P}(u_0, \Omega^{\prime})$ such that $\dist(x_0, \mathfrak{F}(u_0, \Omega^{\prime}))\leq\dist(\Omega^{\prime}, \partial \Omega)$. Then there exists a constant $C>0$ independent of $\varepsilon$ such that
\begin{equation}\label{ocontrol}
	C^{-1} \dist(x_0,\mathfrak{F}(u_0, \Omega^{\prime})) \le u_{0}(x_0) \le C\, \dist(x_0,\mathfrak{F}(u_0, \Omega^{\prime})).
\end{equation}
\end{theorem}

\begin{proof}
From Corollary \ref{c4.1} we know that there exists $y_{\varepsilon}\in\Omega_{\varepsilon}$ such that
$$
    d_{\varepsilon}(x)=|x-y_{\varepsilon}| \,\, \mbox{and} \,\, u^{\varepsilon}(x) \ge c\, d_{\varepsilon}(x) = c\, |x-y_{\varepsilon}|,
$$
for some constant $c>0$ independent of $\varepsilon$. Passing to a subsequence, if necessary, we get for $y_{\varepsilon} \to y_0 \in \mathfrak{F}(u_0, \Omega^{\prime})$
$$
	u_{0}(x) \ge c\,|x_0 - y_0| \ge c\, \dist(x,\mathfrak{F}(u_0, \Omega^{\prime})).
$$
Finally, the upper bound is a consequence of the local Lipschitz estimate for $u_0$.
\end{proof}

The next theorem is an immediate consequence of Theorem \ref{t4.2} as $\varepsilon \to 0^{+}$.

\begin{theorem}\label{limite1} Let $\Omega^{\prime} \Subset \Omega$. For any $x_0 \in \mathfrak{P}(u_0, \Omega^{\prime})$ such that $\dist(x_0, \mathfrak{F}(u_0, \Omega^{\prime}))\leq\dist(\Omega^{\prime}, \partial \Omega)$, there exist constants $C_0>0$ and $r_0>0$ independent of $\varepsilon$, such that
$$
C_0^{-1} r \le \sup_{B_{r}(x_0)} u_0 \le C_0( r + u_0(x_0))
$$
provided $r \le r_0$.
\end{theorem}
The following result shows that, in Hausdorff distance, $\Omega_{\varepsilon}$ converges to $\mathfrak{P}(u_0, \Omega^{\prime})$ as $\varepsilon\to 0^{+}$.

\begin{theorem} \label{limite2} Let $\Omega^{\prime} \Subset \Omega$. Then for a $C_1>1$, the following inclusions hold:
\begin{equation}\nonumber
	\mathfrak{P}(u_0, \Omega^{\prime}) \subset \mathcal{N}_{\delta}(\{u^{\varepsilon_j} > C_1 \varepsilon_j\}) \cap \Omega' \,\, \mbox{and} \,\,
\{u^{\varepsilon_j} > C_1 \varepsilon_j\} \cap \Omega' \subset \mathcal{N}_{\delta}(\{u_0 > 0\}) \cap \Omega',
\end{equation}
provided $\varepsilon_j\leq\delta \ll 1$.
\end{theorem}

\begin{proof} We prove the first inclusion (the other one can be obtained in a similar way). Suppose that it is not true. Then there exists a $\delta_0>0$ such that for every $\varepsilon_j\to 0$ and $\forall x_j \in \mathfrak{P}(u_0, \Omega^{\prime})$
\begin{equation}\label{dist}
\dist(x_j,\{u^{\varepsilon_j}>C_1\varepsilon_j\})>\delta_0.
\end{equation}
For some $y\in \overline{B_{\frac{\delta_0}{2}}(x_j)} \cap \{u^{\varepsilon_j}>C_1\varepsilon_j\}$ we have from Theorem \ref{limite1}
$$
u^{\varepsilon_j}(y) = \sup_{B_{\frac{\delta_0}{2}}(x_j)}u^{\varepsilon_j}(x_j) \geq \frac{1}{2}\sup_{B_{\frac{\delta_0}{2}}(x_j)}u_0(x_j) \geq c \delta_0 \geq C_1\varepsilon_j,
$$
which contradicts \eqref{dist}.
\end{proof}

\begin{theorem} \label{non_deg}
Given $\Omega^{\prime} \Subset \Omega$, there exist constants $C>0$ and $\rho_0>0$, depending only on $\Omega^{\prime}$ and universal parameters, such that for any $x_0 \in \mathfrak{F}(u_0, \Omega^{\prime})$ there holds
\begin{equation}\label{undeg}
	\displaystyle C^{-1}\rho \le \fint\limits_{\;\partial B_{\rho}(x_0)}u_0(x)  \; d \mathcal{H}^{n-1} \le C\,\rho.
\end{equation}
provided $\rho \leq  \rho_0$.
\end{theorem}

\begin{proof}
The upper bound follows from the Lipschitz regularity of $u_0$. The lower bound is a consequence of the nondegeneracy.
\end{proof}

\begin{remark}\label{r6.2}
Repeating the steps of the proof of Theorem \ref{ThmHarIneq} one can show that the Harnack inequality is true for $u_0$ in touching balls. Furthermore, as a consequence of the non-degeneracy and the growth rate, one can prove (as it was done in Theorem \ref{t5.2}) that the free boundary $\mathfrak{F}(u_0)$ is a porous set.
\end{remark}

Next, we prove several geometric-measure properties for $\mathfrak{F}(u_0)$. The ultimate goal is to prove the local finiteness of the $(n-1)$-dimensional Hausdorff measure of the limiting level surface.

First we see that the set $\{u_0>0\}$ has uniform density along $\mathfrak{F}(u_0)$.

\begin{theorem}\label{densidade4}
Let $\Omega^{\prime}\Subset \Omega$. There exists a constants $c_0>0$ such that for any $x_0 \in \mathfrak{F}(u_0, \Omega^{\prime})$ there holds
\begin{equation}\label{den}
\mathfrak{D}(u_0, B_{\rho}(x_0)) \ge c_0	
\end{equation}
provided $\rho \ll 1$. In particular, $\mathscr{L}^{n}(\mathfrak{F}(u_0))=0$.
\end{theorem}

\begin{proof}
The estimate \eqref{den} follows as in the proof of Corollary \ref{c4.2}. We conclude the resuly by using Lebesgue differentiation theorem and a covering argument (Besicovitch-Vitali type theorem, see \cite{EG}).
\end{proof}
\begin{theorem}\label{limit3}
Let $\Omega^{\prime} \Subset \Omega$. There exists a constant $C>0$, depending only on $\Omega^{\prime}$ and universal parameters such that, for any $x_0 \in \mathfrak{F}(u_0, \Omega^{\prime})$, there holds
$$
\mathcal{H}^{n-1} (\mathfrak{F}(u_0, \Omega^{\prime}) \cap B_{\rho}(x_0)) \le C \rho^{n-1}.
$$
\end{theorem}
\begin{proof}
From Theorem \ref{limite2}, for $j \gg 1$ one has
$$
\left[ \mathcal{N}_{\delta}(\mathfrak{F}(u_0, \Omega^{\prime})) \cap B_{\rho}(x_0) \right] \subset  \left[ \mathcal{N}_{4 \delta}(\partial \{u^{\varepsilon_{j}} > C_1 \varepsilon_j\}) \cap B_{2 \rho}(x_0) \right].
$$
Assuming $\varepsilon_j \leq\delta\leq\rho \ll \dist(\Omega^{\prime} , \partial \Omega)$, the hypotheses
of Theorem \ref{Hausdf} are fulfilled, implying the following estimate for the $\delta$-neighborhood,
$$
	\mathcal{L}^{\,n}(\mathcal{N}_{\delta}(\mathfrak{F}(u_0, \Omega^{\prime})) \cap B_{\rho}(x_0)) \le C\delta \rho^{n-1}.
$$
Now, let $\{B_j\}_{j \in \N}$ be a covering of $\mathfrak{F}(u_0, \Omega^{\prime}) \cap B_{\rho}(x_0)$ by balls with radii $\delta>0$ and centered at free boundary points on $\mathfrak{F}(u_0, \Omega^{\prime}) \cap B_{\rho}(x_0)$. Then
$$
	\bigcup_{j} B_j \subset \mathcal{N}_{\delta}(\mathfrak{F}(u_0, \Omega^{\prime})) \cap B_{\rho + \delta}(x_0).
$$
Therefore, there exists a constant $\overline{C}>0$ with universal dependence such that
\begin{eqnarray*}
\displaystyle \mathcal{H}^{n-1}_{\delta}(\mathfrak{F}(u_0, \Omega^{\prime}) \cap B_{\rho}(x_0)) &\le& \overline{C} \sum_{j} \mathcal{L}^{n-1}(\partial B_j)\\
&= & n\frac{\overline{C}}{\delta} \mathcal{L}^{n}(B_j)\\
&\le& n\frac{\overline{C}}{\delta} \mathcal{L}^{n} (\mathcal{N}_{\delta}(\mathfrak{F}(u_0, \Omega^{\prime})) \cap B_{\rho+\delta}(x_0))\\
&\le& C(n) (\rho + \delta)^{n-1}\\
& =& C(n) \rho^{n-1} + o(\delta).
\end{eqnarray*}
Letting $\delta \to 0^{+}$ we finish the proof.
\end{proof}

As an immediate consequence of Theorem \ref{limit3} we conclude that $\mathfrak{F}(u_0)$ has locally finite perimeter. Moreover, the reduced free boundary $\mathfrak{F}^{\star}(u_0) \defeq \partial_{\textrm{red}} \{u_0 >0\}$ has a total $\mathcal{H}^{n-1}$ measure in the sense that $\mathcal{H}^{n-1}(\mathfrak{F}(u_0) \setminus \mathfrak{F}^{\star}(u_0)) =0$ (Theorem 6.7 in \cite{ART}). In particular, the free boundary has an outward vector for $\mathcal{H}^{n-1}$ almost everywhere in $\mathfrak{F}^{*}(u_0)$.

\section{Final comments}\label{r5.2}

\hspace{0.3cm}We finish the paper by analysing the one-dimensional profile representing the corresponding free boundary condition. Let
\begin{equation}\label{5.6}
u^{\varepsilon}_{xx}(u^{\varepsilon}_{x})^2 = \zeta_{\varepsilon}(u^{\varepsilon}) \quad \mbox{in} \quad (-1, 1),
\end{equation}
where $\zeta_{\varepsilon}$ given by
$$
  \zeta_{\varepsilon}(s) = \frac{1}{\varepsilon}\zeta\left(\frac{s}{\varepsilon}\right)
$$
is a high energy activation potential, i.e., a non-negative smooth function supported in $[0, \varepsilon]$. The limiting configuration satisfies (in the viscosity sense)
$$
\Delta_{\infty} u_0 = 0 \quad \mbox{in} \quad \{u_0>0\}\cap (-1, 1).
$$
Multiplying \eqref{5.6} by $u^{\varepsilon}_{x}$ we get
\begin{equation}\label{5.7}
u^{\varepsilon}_{xx}(u^{\varepsilon}_{x})^3 = \zeta_{\varepsilon}(u^{\varepsilon}).u^{\varepsilon}_{x} = \frac{d}{dx}\Xi_{\varepsilon}(u^{\varepsilon}),
\end{equation}
where
$$
\displaystyle \Xi_{\varepsilon}(t) = \int_{0}^{\frac{t}{\varepsilon}} \zeta(s)ds \rightarrow \left(\int \zeta(s)ds\right)\chi_{\{t>0\}}
$$
as $\varepsilon \to 0^+$, i.e.,
$$
\displaystyle \Xi_{\varepsilon}(u^{\varepsilon}) \rightarrow \int \zeta(s)ds, \quad \mbox{as} \quad \varepsilon \to 0^+
$$
provided $u_0(x)>0$. Using change of variable
$$
u^{\varepsilon}_{x}(x)=w,
$$
we re-write
$$
\displaystyle \int \frac{d}{dx}\Xi_{\varepsilon}(u^{\varepsilon})  = \int (u^{\varepsilon})^3_{x} u^{\varepsilon}_{xx}dx = \int w^3dw.
$$
Hence, by computing the anti-derivatives at \eqref{5.7} and letting $\varepsilon \to 0^+$ we obtain the following characterization for limiting condition
$$
\displaystyle |u^{\prime}_0| = \sqrt[4]{4\int \zeta(s)ds} \quad \mbox{on} \quad \partial \{u_0>0\}.
$$
Therefore, the corresponding one-dimensional limiting free boundary problem is given by
$$
\left\{
\begin{array}{rclcl}
\Delta_{\infty} u_0  &=&  0 & \mbox{in} & \{u_0>0\} \cap (-1, 1),\\
u_0 & = & 0 & \mbox{in} & \partial \{u_0>0\},\\
\displaystyle |u^{\prime}_0|  & = & \sqrt[4]{4\int \zeta(s)ds} & \mbox{on} & \partial \{u_0>0\}.
\end{array}
\right.
$$
Furthermore, if for some direction $x_i$ we have
$$
u^{\varepsilon}_{x_ix_i}(u^{\varepsilon}_{x_i})^2 \leq \zeta_{\varepsilon}(u^{\varepsilon}) \quad \mbox{in} \quad \Omega,
$$
then by repeating the previous argument (since $u^{\varepsilon}$ is increasing in direction $x_i$), we conclude
$$
\left|\frac{\partial u_0}{\partial x_i}\right| \leq \sqrt[4]{4\int \zeta(s)ds} \quad \mbox{on} \quad \partial \{u_0>0\}
$$
in every regular point of the free boundary.

\section*{Acknowledgements}
\addcontentsline{toc}{section}{Acknowledgments}
\hspace{0.6cm} This work was partially supported by CNPq (Ci\^{e}ncia sem Fronteiras) and by FCT - SFRH/BPD/92717/2013, as well as by ANPCyT - PICT 2012-0153 and by CONICET-
Argentina. We would like to thank Noem\'{i} Wolanski and Eduardo Teixeira for several insightful comments and discussions. JVS thanks research group of PDEs of Universidad de Buenos Aires-UBA for fostering a pleasant and productive scientific atmosphere during his postdoctoral program.

\end{document}